\documentclass[11pt]{amsart}
\usepackage{amsmath}
\usepackage{amssymb}
\usepackage{amscd}

\def\NZQ{\mathbf}               % the font for N,Z,Q,R,C

\def\CC{{\NZQ C}}

\theoremstyle{definition}

\theoremstyle{remark}

\newtheorem{Theorem}{Theorem}[section]
\newtheorem{Lemma}[Theorem]{Lemma}
\newtheorem{Corollary}[Theorem]{Corollary}
\newtheorem{Proposition}[Theorem]{Proposition}

\let\epsilon\varepsilon
\let\phi=\varphi
\let\kappa=\varkappa
\begin{document}

\title{Associated Primes of the Square of the Alexander Dual of Hypergraphs}

\author{Ashok Cutkosky}
\begin{abstract} 
The purpose of this paper is to provide methods for determining the associated primes of $(I(H)^\vee)^2$ for an m-hypergraph $H$. We prove a general method for detecting associated primes of the square of the Alexander dual of the edge ideal based on combinatorial conditions on the m-hypergraph. Also, we demonstrate a more efficient combinatorial criterion for detecting the non-existence of non-minimal associated primes. In investigating 3-hypergraphs, we prove a surprising extension of the previously discovered results for 2-hypergraphs (simple graphs). For 2-hypergraphs, associated primes of the square of the Alexander dual of the edge ideal are either of height 2 or of odd height greater than 2. However, we prove that in the 3-hypergraph case, there is no such restriction - or indeed any restriction - on the heights of the associated primes. Further, we generalize this result to any dimension greater than 3. Specifically, given any integers $m$, $q$, and $n$ with $3\leq m\leq q\leq n$, we  construct a $m$-hypergraph of size $n$ with an associated prime of height $q$. We further prove that it is possible to construct  connected $m$-hypergraphs under the same conditions.

\end{abstract}

\maketitle

\section*{Introduction}

An $m$-hypergraph is a vertex set $\{x_1,...,x_n\}$ and an edge set $$\{\{x_{i_1},...,x_{i_m}\},\{x_{j_1},...,x_{j_m}\},...\}$$ where each edge joins $m$ vertices. Many combinatorial properties of hypergraphs have been defined and studied. However, it is possible, and often advantageous, to study hypergraphs from an algebraic rather than a combinatorial point of view. This is done by associating to the vertex set a polynomial ring in $n$ variables over a field and associating to the edge set an ideal, $I$, called the edge ideal in that ring. The edge ideal is the ideal generated by monomials $x_{i_1}\cdots x_{i_m}$ where $\{x_{i_1},...,x_{i_m}\}$ is an edge. There is a further function called the Alexander dual that transforms the edge ideal of a hypergraph into another ideal, $I^\vee$, obtained via intersection of ideals of the form $(x_{i_1},...,x_{i_m})$ where $\{x_{i_1},...,x_{i_m}\}$ is an edge. This new ideal has the interesting property that the minimal generators correspond exactly to the minimal vertex covers of the hypergraph. We can also take symbolic powers of the Alexander dual, which have slightly different properties. In particular, the symbolic $k$th power of the Alexander dual is generated by the monomials corresponding to the $k$-covers of the hypergraph, while the ordinary $k$th powers of the Alexander dual are generated by the $k$-covers that are sums of $1$-covers, a phenomenon that has been studied by Herzog, Hibi and Trung \cite{HHT}, Gitler, Reyes and Villarreal \cite{GRV}, Fransisco, H\`a and Van Tuyl \cite{FHT} and other mathematicians. Herzog, Hibi and Trung prove in \cite{HHT} that  the algebra of vertex covers of a hypergraph (the symbolic algebra of the Alexander dual) is finitely generated. It is proven in \cite{FHT} that in the case of 2-hypergraphs (or simple graphs), the associated primes of $(I^\vee)^2$ correspond to the edges of the graph, and to the odd induced cycles. 

We prove in Theorem \ref{multd} of this paper that for any integers $n$, $q$ and $m$ with $n\geq q\geq m\geq3$, there exists an $m$-hypergraph of any size $n$ such that the square of the Alexander dual of the edge ideal of the $m$-hypergraph has an associated prime of height $q$. We also prove that for $m$-hypergraphs, it is not possible to have an associated prime of height less than $m$. Further, we have a generalized statement (Theorem \ref{indt}) for m-hypergraphs that an associated prime for an induced subhypergraph is an associated prime for the full hypergraph. Using the above statements, one can prove that for any dimension higher than 2, there are connected hypergraphs of any size with any possible height of associated prime. These results are in stark contrast with the behavior of 2-hypergraphs, where it is not possible to obtain associated primes of even height greater than 2 (Theorem 1.1 \cite{FHT}). Also, we have in Theorem \ref{finder} a necessary and sufficient combinatorial criterion for determining the existence of and an explicit description for the associated primes. Briefly, the method involves finding a 2-cover that is not the sum of two 1-covers - this represents an element not in $(I(H)^\vee)^2$. Next, we note which, if any, components it is possible to increase by 1 to make a 2-cover that is the sum of 1-covers, and prove that adding anything to any other component does not make a sum of two 1-covers. Then the components noted in the second step correspond to the generators of an associated prime. This method provides a nice combinatorial gateway between the structure of the hypergraph and the structure of the algebra.

While we have found numerous examples to show that the elements of a cycle do not necessarily generate an associated prime, it is an interesting open question as to whether the converse is true. Further, although in many examples the associated primes of higher powers of the Alexander dual appear to be the same as for the second power, this is not true in general. Two examples due to Christopher Francisco in which the third power of the Alexander dual has a number of associated primes which are not associated primes for the square are the 3-hypergraph with edge ideal $$(x_2x_4x_8, x_3x_5x_6, x_4x_7x_9, x_6x_8x_9,x_3x_4x_7, x_1x_2x_9, x_1x_4x_6, x_1x_8x_9, x_2x_5x_9,x_3x_4x_6)$$ and the complete 2-hypergraph on five vertices.

I would like to thank Dr. Christopher Francisco for his guidance in this project.

\section{Tools for Detecting Associated Primes}

An m-hypergraph, $H$, is a vertex set $V=\{x_1,x_2,...,x_n\}$ and an edge set $E$ such that the elements of $E$ are subsets of $V$ containing $m$ distinct elements. The size of $H$ is said to be the size of $V$. A 2-hypergraph is a simple graph (a graph with no loops or multiple edges).

The edge ideal $I$ of an m-hypergraph $H$ of size $n$ with edge set $E$ and vertex set $V$ is the ideal generated by the monomials formed by multiplying all the vertices that make up an edge in a polynomial ring in $n$ variables over a field.
The Alexander dual of an edge ideal $I$ is 
$$
I^\vee=\bigcap_{\{x_{i_1},x_{i_2},...,x_{i_m}\}\in E}(x_{i_1},x_{i_2},...,x_{i_m}).
$$

In a 2-hypergraph with $n$ vertices, a $k$-cover is a non-zero element of $\textbf{N}^n$, $(a_1,a_2,...,a_n)$, such that for every edge $\{x_i,x_j\}$, $a_i+a_j\geq k$. The definition extends nicely to m-hypergraphs: for every edge $\{x_{i_1},x_{i_2},...,x_{i_m}\}$ we have $a_{i_1}+a_{i_2}+\cdots+a_{i_m}\geq k$. Note that any $k$-cover is a 0-cover.
If a $k$-cover $A$ can be written as $A=B+C$ where $B$ is a $p$-cover and $C$ is a $(k-p)$-cover, then the $k$-cover $A$ is said to be reducible. Otherwise, $A$ is said to be irreducible.
A $k$-cover $(a_1,a_2,...,a_n)$ corresponds to the monomial $x_1^{a_1}x_2^{a_2}\cdots x_n^{a_n}$. For a $k$-cover $A$, we will call the corresponding monomial $\overline{A}$. Note that for any monomial $Q$, there is a $k\ge 0$ and a $k$-cover $P$ such that $Q=\overline{P}$. We also have the observation that for $k$-covers $A$ and $B$, $\overline{A+B}=(\overline{A})(\overline{B}$).

\begin{Lemma}\label{lm1}
For all $k\geq 1$, the symbolic $k$th power of the Alexander dual, defined by
$$
(I^\vee)^{(k)}=\bigcap_{\{x_{i_1},x_{i_2},...,x_{i_m}\}\in E}(x_{i_1},x_{i_2},...,x_{i_m})^k
$$
is generated by the the monomials corresponding to $k$-covers.
\begin{proof}
Since the intersection of powers of monomial ideals is a monomial ideal, it suffices to show that any monomial in $(I^\vee)^{(k)}$ corresponds to a $k$-cover and vice versa. Let $M\in (I^\vee)^{(k)}$ be a monomial. Since $M$ is a monomial, there exists a $0$-cover $P=(p_1,...,p_n)$ such that $M=\overline{P}$. Since $M\in (I^\vee)^{(k)}$, $M\in(x_{i_1},x_{i_2},...,x_{i_m})^k$ for all edges $\{x_{i_1},x_{i_2},...,x_{i_m}\}\in E$. Thus for all edges $\{x_{i_1},x_{i_2},...,x_{i_m}\}\in E$, the sum of the exponents on $x_{i_1},x_{i_2},...,x_{i_m}$ must be at least $k$. Since $M=x_1^{p_1}\cdots x_n^{p_n}$, we have $p_{i_1}+p_{i_2}+\cdots+p_{i_m}\geq k$. Thus $P$ is a $k$-cover. Reversing the argument, we see that for every $k$-cover $P$, $\overline{P}\in (I^\vee)^{(k)}$.
\end{proof}

\end{Lemma}

Thus the following facts are true: The monomials corresponding to irreducible 1-covers are minimal generators for $I^\vee$, and
the monomials corresponding to 2-covers are generators for $(I^\vee)^{(2)}$. Similarly
the monomials corresponding to 2-covers that are the sums of 1-covers are generators of $(I^\vee)^2$.

By localizing at prime ideals, one can see that $(x_{i_1},x_{i_2},...,x_{i_m})$ is a minimal prime of $(I^\vee)^2$ if and only if $\{x_{i_1},x_{i_2},...,x_{i_m}\}\in E$.
If $P_1,P_2,...,P_r$ are the other (possibly non-existent) associated primes of $(I^\vee)^2$ (embedded primes),
\begin{equation}\label{*}
(I^\vee)^2=\left[\bigcap_{\{x_{i_1},x_{i_2},...,x_{i_m}\}\in E}(x_{i_1},x_{i_2},...,x_{i_m})^2\right]\cap Q_1\cap Q_2 \cap... \cap Q_r
\end{equation}
where $Q_i$ is $P_i$-primary.

%\begin{Theorem}\label{thm1}
%$(x_{i_1},...,x_{i_j})$ is an associated prime of $(I^\vee)^2$ if and only if there exists a monomial $M$ such that a monomial $N$ satisfies $NM$ is a %sum of 1-covers if and only if $N\in (x_{i_1},...,x_{i_j})$.
%\end{Theorem}

Suppose that $H$ is an m-hypergraph.
A set of vertices $S$ is an independent set if no edge contains only vertices in $S$.
The neighborhood of an independent set $S$, $N(S)$, is the set of vertices $x_i\notin S$ such that there is an edge $\{x_i,x_{j_1},...,x_{j_m}\}$ such that all of the $x_{j_l}$ are in $S$.
Suppose $S$ is an independent set, and $T$ is the neighborhood of $S$. Define $A_s=(a_1,...,a_n)$ where $a_i=0$ if $x_i\in S$, $a_i=2$ if $x_i\in T$ and $a_i=1$ if $x_i\notin S\cup T$.

\begin{Proposition}\label{prop5} Let $H$ be an m-hypergraph.
\begin{enumerate}
\item [1.] Suppose $S$ is an independent set in $H$. Then $A_s$ is a 2-cover.
\item [2.] Suppose $W$ is a 2-cover, such that $W\neq A_s$ for all independent sets $S$. Then there exists a 2-cover $X=A_s$ for some independent set $S$ and a 0-cover Y such that $W=X+Y$.
\end{enumerate}
\end{Proposition}
\begin{proof}
Let $H$ be an m-hypergraph of size $n$ with vertex set $V$ and edge set $E$.

1. Suppose $S$ is an independent set in $H$. Then we have $A_s=(a_1,...,a_m)$ where $a_i=0$ for $x_i\in S$, $a_i=2$ for $x_i\in N(S)$ and $a_i=1$ for $x_i\in V-S-N(S)$. Let $\{x_{i_1},...,x_{i_m}\}$ be an edge in $H$. We must show that $a_{i_1}+a_{i_2}+\cdots+a_{i_m}\geq 2$. Since $S$ is independent, at least one of $x_{i_1},...,x_{i_m}$ is not in $S$. Since $a_i=0$ implies $x_i\in S$, at least one of $a_{i_1},a_{i_2},...,a_{i_m}$ is non-zero. If more than one of these is non-zero, we are done. Suppose exactly one of these, say $a_{i_k}$, is non-zero. Then we have $x_{i_j}\in S$ for $j\neq k$. Thus $x_{i_k}\in N(S)$, so $a_{i_k}=2$, and $a_{i_1}+a_{i_2}+\cdots+a_{i_m}=2\geq 2$. So $A_s$ is a 2-cover.

2. Suppose $W=(w_1,w_2,...,w_n)$ is a 2-cover, $W\neq A_s$ for all independent sets $S$. Let $S=\{x_i\mid w_i=0\}$. Then $S$ is an independent set since if there was an edge $\{x_{i_1},...,x_{i_m}\}$ such that $\{x_{i_1},...,x_{i_m}\}\subset S$, then $w_{i_1}+...+w_{i_m}=0$, and this cannot happen since $W$ is a 2-cover.
We will show that $W-A_s$ is a 0-cover.
Let $x_j\in N(S)$. Then there is an edge $\{x_j,x_{i_2},...,x_{i_m}\}$ such that $\{x_{i_2},...,x_{i_m}\}\subset S$. Thus $w_{i_k}=0$ for $2\leq k \leq m$. Since $W$ is a 2-cover, $w_j=w_j+w_{i_2}+\cdots+w_{i_m}\geq 2$. Further, by our definition of $S$, $w_i\geq 1$ for all $x_i\notin S$. Thus $W-A_s\in \textbf{N}^n$. Also, since $W\neq A_s$, $W-A_s\neq \textbf{0}$. Since any nonzero element of $\textbf{N}^n$ is a 0-cover, we have $W-A_s$ is a 0-cover. $A_s$ is a 2-cover by part 1, so we have $W=A_s+(W-A_s)$ which proves the proposition.
\end{proof}

\begin{Theorem}\label{b}
Let $H$ be an m-hypergraph with edge ideal $I$. Then
$(I^\vee)^2$ has no non-minimal associated primes if and only if for every independent set $S\subset V$, $A_s$ is a sum of two 1-covers.
\end{Theorem}

\begin{proof}
Let $H$ be an m-hypergraph with edge ideal $I$.

Suppose $(I^\vee)^2$ has no non-minimal associated primes.
Then by equation (\ref{*}),
$$
(I^\vee)^2=\bigcap_{\{x_{i_1},x_{i_2},...,x_{i_m}\}\in E}(x_{i_1},x_{i_2},...,x_{i_m})^2.
$$
But we also have
$$
(I^\vee)^{(2)}=\bigcap_{\{x_{i_1},x_{i_2},...,x_{i_m}\}\in E}(x_{i_1},x_{i_2},...,x_{i_m})^2,
$$
so $(I^\vee)^2=(I^\vee)^{(2)}$.
Further, we know that $(I^\vee)^2$ is generated by the 2-covers that are the sums of 1-covers and $(I^\vee)^{(2)}$ is generated by the 2-covers. Thus the 2-covers are generated by the 2-covers that are the sums of 1-covers. Thus if $\overline{A}$ is a monomial corresponding to a 2-cover, $\overline{A}=Q\overline{B}$ where $B$ is a 2-cover and $Q$ is another (non-zero) monomial. Thus we can write $A=Q\overline{B_1}\overline{B_2}$ where $B_1$ and $B_2$ are 1-covers. Further, $Q\overline{B_1}$ must also correspond to a 1-cover, so $A$ is the sum of two 1-covers.

Now suppose that for every independent set $S$, $A_s$ is a sum of two 1-covers. Let $B$ be a 2-cover. Then by proposition \ref{prop5}, $B=A_s$ for some independent set $S$, or $B=W+A_s$ for some 0-cover $W$ and some independent set $S$. Whichever is the case, we can say that $A_s=Q_1+Q_2$ for two 1-covers $Q_1$ and $Q_2$. So $B=Q_1+Q_2$ or $B=(W+Q_1)+Q_2$. Since $W+Q_1$ is a 1-cover as well, $B$ can be written as the sum of two 1-covers. Thus the set of 2-covers equals the set of 2-covers that can be written as the sum of two 1-covers.
Therefore, $(I^\vee)^2=(I^\vee)^{(2)}$. So
$$
(I^\vee)^2=\bigcap_{\{x_{i_1},x_{i_2},...,x_{i_m}\}\in E}(x_{i_1},x_{i_2},...,x_{i_m})^2.
$$

Since $(x_{i_1},x_{i_2},...,x_{i_m})^2$ is $(x_{i_1},x_{i_2},...,x_{i_m})$-primary, we have the irredundant primary decomposition, so the only associated primes of $(I^\vee)^2$ are the minimal primes.
\end{proof}

We also have the following result, which can be invaluable in detecting associated primes:
\begin{Theorem}\label{finder}
Let $H$ be an m-hypergraph of size $n$. Then an ideal $P\subset K[x_1,...,x_n]$ is a non-minimal associated prime of $(I(H)^\vee)^2$ if and only if $P=(x_{i_1},...,x_{i_q})$ and there exists a 2-cover $C$ that is not the sum of two 1-covers such that $x_{i_j}\overline{C}$ corresponds to a 2-cover that is the sum of two 1-covers for any $j\leq q$, and for any $W\notin (x_{i_1},...,x_{i_q})$, $W\overline{C}$ does not correspond to a 2-cover that is the sum of two 1-covers.
\end{Theorem}
\begin{proof}
Suppose the ideal $P\subset K[x_1,...,x_n]$ is a non-minimal associated prime of $(I(H)^\vee)^2$. Since $(I(H)^\vee)^2$ is a monomial ideal, $P=(x_{i_1},...,x_{i_q})$. Then there exists some monomial $Z$ such that $((I(H)^\vee)^2:Z)=P$. Since $Z$ is a monomial, there exists some 0-cover $C=(c_1,...,c_n)$ such that $Z=\overline{C}$. Note that $x_{i_j}Z\in(I(H)^\vee)^2$ for all $j\leq q$. Since $(I(H)^\vee)^2$ is generated by the monomials corresponding to 2-covers that are the sum of two 1-covers, $x_{i_j}Z$ corresponds to a 2-cover that is the sum of two 1-covers for all $j$. Further, since $Z$ cannot be in $(I(H)^\vee)^2$, $C$ cannot be the sum of two 1-covers. Since $P$ is non-minimal, $P$ has height at least $m+1$. Let $R=(r_1,...,r_n)$ be the 0-cover such that $\overline{R}=x_{i_1}Z$. Then $R$ is the sum of two 1-covers. Suppose that $A=\{x_{a_1},...,x_{a_m}\}$ is an edge such that $x_{i_1}\notin A$. Then we have $r_{a_1}+\cdots+r_{a_m}\geq 2$. But for $i\neq {i_1}$, $r_i=c_i$. Thus $c_{a_1}+\cdots+c_{a_m}\geq 2$. Suppose that $B=\{x_{b_1},...,x_{b_m}\}$ is an edge where $b_1=i_1$. Then, since $P$ has height greater than $m$, there exists $k$ such that $k=i_j$ for some $j$ and $k\neq b_j$ for all $j$. Let $S=(s_1,...,s_n)$ be the 0-cover such that $\overline{S}=x_{k}Z$. Then $S$ is the sum of two 1-covers. Thus $s_{b_1}+\cdots+s_{b_m}\geq 2$. But $s_i=c_i$ for all $i\neq k$. Since $k\neq b_j$ for all $j$, $c_{b_1}+\cdots+c_{b_m}\geq 2$. Thus $C$ is a 2-cover. Let $W\notin P$. Then we have $WZ\notin (I(H)^\vee)^2$. Thus $W\overline{C}=WZ$ does not correspond to a 2-cover that is the sum of two 1-covers.

Now suppose that $C=(c_1,...,c_n)$ is a 2-cover that is not the sum of two 1-covers, and that $x_{i_j}\overline{C}$ corresponds to a 2-cover that is the sum of two 1-covers for all $j\leq q$ and for any $W\notin (x_{i_1},...,x_{i_q})$, $W\overline{C}$ does not correspond to a 2-cover that is the sum of two 1-covers.  Since $(I(H)^\vee)^2$ is generated by the monomials that correspond to 2-covers that are the sum of two 1-covers, $\overline{C}\notin (I(H)^\vee)^2$ and $x_{i_j}\overline{C}\in (I(H)^\vee)^2$ for $j\leq q$. Thus $(x_{i_1},...,x_{i_q})\subset ((I(H)^\vee)^2:\overline{C})$. Further, we have for $W\notin (x_{i_1},...,x_{i_q})$, $W\overline{C}\notin (I(H)^\vee)^2$. Thus $(x_{i_1},...,x_{i_q}) = ((I(H)^\vee)^2:\overline{C})$. Since $(x_{i_1},...,x_{i_q})$ is a prime ideal, $(x_{i_1},...,x_{i_q})$ is an associated prime of $(I(H)^\vee)^2$.
\end{proof}
The incidence matrix of a hypergraph with $v$ vertices and $e$ edges is the $v\times e$ matrix with $a_{ij}=1$ if $x_i$ is a vertex of the $j$th edge, and $a_{ij}=0$ otherwise.
If there does not exist a square submatrix with an odd number of columns such that each row and column has exactly 2 ones, then the hypergraph is called balanced.
It is shown in Proposition 4.11 of \cite{GRV} that a balanced hypergraph has no non-minimal associated primes. In the case of a 2-hypergraph (a simple graph), balanced is equivalent to bipartite.

This condition is sufficient to show that a hypergraph has no non-minimal associated primes, but it is not necessary, as shown by this example, of
the 3-hypergraph given by the edge ideal $I=(x_1x_2x_3,x_3x_4x_5,x_5x_6x_1,x_2x_3x_4)$, which  is not balanced, but has no non-minimal associated primes.

In the case of 2-hypergraphs, there is a simple necessary and sufficient condition for $(I^\vee)^2$ to have no nonminimal associated primes, which we state in Theorem \ref{c} below, with a very simple and direct proof. Theorem \ref{c} is also a consequence of facts proven in Theorem 5.1 of \cite{HHT}.

A 2-hypergraph with vertex set $V$ and edge set $E$ is called bipartite if there exist two subsets $A\subset V$, $B\subset V$ such that $V=A\cup B$, $A\cap B=\emptyset$ and $\{x_i,x_j\}\in E$ implies either $x_i\in A$, $x_j\in B$ or $x_i\in B$, $x_j\in A$.
\begin{Theorem}\label{c}
Suppose $I$ is the edge ideal of a 2-hypergraph $G$ (a simple graph). Then the following are equivalent:
\begin{enumerate}
\item [1.] $G$ is bipartite

\item [2.] The 2-cover $(1,1,....,1)$ is a sum of two 1-covers

\item [3.] $(I^\vee)^2$ has no non-minimal associated primes.
\end{enumerate}
\end{Theorem}
\begin{proof}
($2\Rightarrow 1$)
Suppose $G$ is a 2-hypergraph of size $n$.
%Then if $G$ has vertex set $V$ and edge set $E$, there exist two subsets $A\subset V$, $B\subset V$ such that $V=A\cup B$, $A\cap B=\emptyset$ and %$\{x_i,x_j\}\in E$ implies either $x_i\in A$, $x_j\in B$ or $x_i\in B$, $x_j\in A$.
%Let $A'=(a_1,a_2,...,a_n)\in\textbf{N}^n$ where $a_i=1$ if $x_i\in A$ and $a_i=0$ otherwise. Similarily, let $B'=(b_1,...,b_n)\in\textbf{N}^n$ where %$b_i=1$ if $x_i\in B$ and $b_i=0$ otherwise.
%Then since each edge in $E$ has one vertex in $A$ and one vertex in $B$, both $A'$ and $B'$ are 1-covers. Further, $A+B=(1,1,1,...,1)$.
Suppose $(1,1,...,1)$ is the sum of two 1-covers, $Q$ and $P$. Then each component of $Q$ and $P$ must be either $1$ or $0$. Further, if $q_i=0$, $p_i=1$ and vice versa.
Let $A=\{x_i\mid q_i=1\}$ and $B=\{x_i\mid p_i=1\}$. Then we have $V=A\cup B$ and $A\cap B=\emptyset$. Further, since $Q$ and $P$ are both 1-covers, each edge must have at least one vertex in $A$ and $B$. Thus $G$ is bipartite.

($1\Rightarrow 3$)
Suppose $G$ is bipartite. Let $A\subset V$, $B\subset V$ such that $V=A\cup B$, $A\cap B=\emptyset$ and $\{x_i,x_j\}\in E$ implies either $x_i\in A$, $x_j\in B$ or $x_i\in B$, $x_j\in A$.
Let $S$ be an independent set in $G$. Let $A'=(a_1,...,a_n)$ where $a_i=0$ if $x_i\in S$, $a_i=1$ if $x_i\in N(S)$, $a_i=1$ if $x_i\in (V-S-N(S))\cap A$, and $a_i=0$ if $x_i\notin (V-S-N(S))\cap A$. Similarily, let $B'=(b_1,...,b_n)$ where $b_i=0$ if $x_i\in S$, $b_i=1$ if $x_i\in N(S)$, $b_i=1$ if $x_i\in (V-S-N(S))\cap B$, and $b_i=0$ if $x_i\notin (V-S-N(S))\cap B$. Then $A+B=A_s$.
Let $\{x_i,x_j\}\in E$. Then $\{x_i,x_j\}$ is not a subset of $S$. Suppose one of $x_i,x_j$ is in $S$. Then the other is in $N(S)$, so $a_i+a_j=b_i+b_j=1$. Now suppose $\{x_i,x_j\}\cap S=\emptyset$. Then $\{x_i,x_j\}\subset V-S-N(S)$, and one of $x_i,x_j$ is in $A$ and the other is in $B$, so we have either $x_i\in (V-S-N(S))\cap A$, $x_j\in(V-S-N(S))\cap B$ or vice versa. Either way, $a_i+a_j=b_i+b_j=1$. Thus $A'$ and $B'$ are both 1-covers, so each $A_s$ is a sum of two 1-covers and by Theorem \ref{b}, $(I^\vee)^2$ has no non-minimal associated primes.

($3\Rightarrow 2$)
Suppose $(I^\vee)^2$ has no non-minimal associated primes. Then by theorem \ref{b}, $A_{\emptyset}=(1,1,1,...,1)$ is the sum of two 1-covers.
\end{proof}

Theorem \ref{b} gives a generalization of Theorem \ref{c} to m-hypergraphs.
However, the combinatorial conditions on Theorem \ref{b} are much more complicated on hypergraphs. For example, the 2-cover $(1,1,...,1)$ can be the sum of two 1-covers when there are non-minimal associated primes. An example of an edge ideal for such a situation is
$$
I=(x_1x_2x_3,x_3x_4x_5,x_5x_6x_1).
$$
In fact, in m-hypergraphs it is difficult to find an example for which $(1,1,...,1)$ is not the sum of two 1-covers. However, the complete 3-hypergraph on 10 vertices has this property (A complete m-hypergraph is an m-hypergraph for which every collection of $m$ vertices is connected by an edge).

An induced subhypergraph on a subset of the vertex set of an m-hypergraph is the m-hypergraph formed by the subset of the vertex set and the subset of the edge set consisting of all edges for which all the vertices on the edge are in the subset of the vertex set.

A theorem that provides some insight into the structure of associated primes of $(I^\vee)^2$ is:
\begin{Theorem}\label{indt}
Let $H$ be an m-hypergraph of size $n$. Then the ideal $(x_{i_1},...,x_{i_q})\subset K[x_1,...,x_n]$ is an associated prime of $(I(H)^\vee)^2$ if and only if the ideal $(x_{i_1},...,x_{i_q})\subset K[x_{i_1},...,x_{i_q}]$ is an associated prime of $(I(H')^\vee)^2$ where $H'$ is the induced subhypergraph on $x_{i_1},...,x_{i_q}$.

%Let $H'$ be an m-hypergraph of size $q$. Then the ideal $(x_{i_1},...,x_{i_q})\subset K[x_{i_1},...,x_{i_q}]$ is an associated prime of $(I(H)^\vee)^2$ if $(x_{i_1},...,x_{i_q})\subset K[x_1,...,x_n]$ is an associated prime of $(I(H)^\vee)^2$ where $H$ is an m-hypergraph for which $H'$ is an induced subhypergraph.%the induced subgraph on $x_{i_1},...,x_{i_q}$.
\end{Theorem}
\begin{proof}
Suppose $(x_{i_1},...,x_{i_q})\subset K[x_{i_1},...,x_{i_q}]$ is an associated prime of $(I(H')^\vee)^2$ where $H'$ is the induced subgraph on $x_{i_1},...,x_{i_q}$. Suppose $q=m$. Then $(x_{i_1},...,x_{i_q})$ is a minimal prime, and $\{x_{i_1},...,x_{i_q}\}$ is an edge in both $H$ and $H'$, so $(x_{i_1},...,x_{i_q})$ is an associated prime in $(I(H)^\vee)^2$. Now suppose $q>m$. Then there exists a 2-cover of $H'$, $C=(c_{i_1},...,c_{i_q})$, such that $C$ is not the sum of two 1-covers, but if we add 1 to the ${i_k}$th component of $C$ for any $k$, we obtain a 2-cover that is the sum of two 1-covers. Let $T=(t_1,...,t_n)$ be such that $t_{i_j}=c_{i_j}$ and $t_l=2$ whenever $l\neq i_j$ for all $j\leq q$. Let $\{x_{p_1},...,x_{p_m}\}$ be an edge in $H$. Suppose $\{x_{p_1},...,x_{p_m}\}$ is also an edge in $H'$. Then $t_{p_1}+\cdots+t_{p_m}=c_{p_1}+\cdots+c_{p_m}\geq 2$ since $C$ is a 2-cover. Suppose at least one of $x_{p_1},...,x_{p_m}$, say $x_{p_j}$, is not in $H'$. Then $t_{p_j}=2$, so $t_{p_1}+\cdots+t_{p_m}\geq 2$. Thus $T$ is a 2-cover of $H$. Suppose $T=D+E$ where $D=(d_1,...,d_n)$ and $E=(e_1,...,e_n)$ are 1-covers of $H$. Then for all edges $\{x_{p_1},...,x_{p_m}\}$ in $H$, $d_{p_1}+\cdots+d_{p_m}\geq 1$ and $e_{p_1}+\cdots+e_{p_m}\geq 1$. Specifically, $D'=(d_{i_1},...,d_{i_q})$ and $E'=(e_{i_1},...,e_{i_q})$ are 1-covers in $H'$. Further, $E'+D'=C$, which cannot happen since $C$ is not the sum of two 1-covers. Thus $T$ is not the sum of two 1-covers. Now we must show that if we add 1 to the ${i_k}$th component of $T$ for any $k$, we obtain a 2-cover that is the sum of two 1-covers, and that if we add 1 to any other component, we do not. Let $F=(f_1,...,f_n)$ where $f_{i_k}=1$ for some $k$, and $f_i=0$ everywhere else. Let $F'=(f_{i_1},...,f_{i_q})$. Then there exist two 1-covers of $H'$, $A'=(a'_{i_1},...,a'_{i_q})$ and $B'=(b'_{i_1},...,b'_{i_q})$ such that $A'+B'=F'+C$. Let $A=(a_1,...,a_n)$, $B=(b_1,...,b_n)$ where $a_{i_j}=a'_{i_j}$, $b_{i_j}=b'_{i_j}$ and $a_i=b_i=1$ everywhere else. Then we have $A$ and $B$ are 1-covers on $H$ and $A+B=F+T$. Now let $F$ be any 0-cover such that $f_{i_k}=0$ for all $k$. Then $F'=(0,...,0)$, and so the same argument that shows that $T$ is not the sum of two 1-covers shows that $F+T$ is not the sum of two 1-covers. Thus the ideal $(x_{i_1},...,x_{i_q})\subset K[x_1,...,x_n]$ is an associated prime of $(I(H)^\vee)^2$.

Now suppose the ideal $(x_{i_1},...,x_{i_q})\subset K[x_1,...,x_n]$ is an associated prime of $(I(H)^\vee)^2$. Suppose $q=m$. Then we $(x_{i_1},...,x_{i_q})$ is minimal and so is an associated prime of $(I(H')^\vee)^2$. Now suppose $q>m$. Then there exists a 2-cover $C=(c_1,...,c_n)$ such that $C$ is not the sum of two 1-covers, but if we add 1 to the ${i_k}$th component of $C$ for any $k$, we obtain a 2-cover that is the sum of two 1-covers, and that if we add anything to any other components, we do not. Let $H'$ be the induced subhypergraph on $\{x_{i_1},...,x_{i_q}\}$. Let $C'=(c'_{i_1},...,c'_{i_q})$ where $c'_{i_j}=c_{i_j}$. Let $\{x_{p_1},...,x_{p_m}\}$ be an edge in $H'$. Then $\{x_{p_1},...,x_{p_m}\}$ is also an edge in $H$. Thus $c_{p_1}+\cdots+c_{p_m}=c'_{p_1}+\cdots+c'_{p_m}\geq 2$. Thus $C'$ is a 2-cover of $H'$. Suppose $C'$ is the sum of two 1-covers. We have $C'=A'+B'$ where $A'=(a'_{i_1},...,a'_{i_q})$, $B'=(b'_{i_1},...,b'_{i_q})$ are 1-covers of $H'$. 
Let $\{x_{z_1},...,x_{z_r}\}$ be the set of vertices not in $H'$. Let $W=(w_1,...,w_n)$ be a 0-cover of $H$ such that $\overline{W}=x_{z_1}^2\cdots x_{z_r}^2\overline{C}$. 
%Then it is clear that for any edge $\{x_{p_1},...,x_{p_m}\}$ not in $H'$, $w_{p_1}+...+w_{p_m}\geq 2$. 
Then, for any vertex $x_j$ not in $H'$, $w_j\geq2$. Let $F=(f_1,...,f_n)$ be such that $f_i=0$ if $x_i$ is in $H'$, $f_i=1$ otherwise. Let $G=(g_1,...,g_n)$ such that $g_i=0$ if $x_i$ is in $H'$, $g_i=w_i-f_i$ otherwise. Then we have for each $x_i$ not in $H'$, $g_i\geq f_i=1$.

Let $A=(a_1,...,a_n)$ where $a_{i_j}=a'_{i_j}$ for $1\leq j \leq q$, $a_k=0$ everywhere else. Similarly, let $B=(b_1,...,b_n)$ where $b_{i_j}=b'_{i_j}$ for $1\leq j\leq q$, $b_k=0$ everywhere else. Then $W=(A+F)+(B+G)$. For each edge $\{x_{p_1},...,x_{p_m}\}$ not in $H'$, $g_{p_1}+\cdots+g_{p_m}\geq f_{p_1}+\cdots+f_{p_m}\geq1$, and for each edge $\{x_{p_1},...,x_{p_m}\}$ in $H'$, $a_{p_1}+\cdots+a_{p_m}\geq 1$ and $b_{p_1}+\cdots+b_{p_m}\geq 1$. Thus $W$ is the sum of two 1-covers. However, this is contrary to the initial supposition about $C$, namely that if we add anything to components outside $H'$, we do not obtain a sum of two 1-covers. Thus $C'$ is not the sum of two 1-covers. Now we must show that for any $x_{i_j}$, $x_{i_j}\overline{C'}$ corresponds to a sum of two 1-covers. By our initial supposition, $x_{i_j}\overline{C}=\overline{E}\overline{T}$ where $E=(e_1,...,e_n)$ and $T=(t_1,...,t_n)$ are 1-covers. Let $E'=(e_{i_1},...,e_{i_q})$ and $T'=(t_{i_1},...,t_{i_q})$. Then $T'$ and $E'$ are 1-covers in $H'$ Further, $\overline{T'}\overline{E'}=x_{i_j}\overline{C'}$. Thus we have $(x_{i_1},...,x_{i_j})$ is an associated prime of $(I(H')^\vee)^2.$
\end{proof}

\section{m-hypergraphs}

In this section we investigate the non-minimal associated primes of $(I^\vee)^2$ in terms of properties of the hypergraph. The motivation is the following theorem for 2-hypergraphs (simple graphs) proven in \cite{FHT}. 

Define a cycle as an ordered set $C = \{x_{i_1},x_{i_2},...,x_{i_n}\}\subseteq V$ such that for $1\leq j < n$, $x_{i_j}$ and $x_{i_{j+1}}$ are connected by an edge on the induced subhypergraph, and $x_{i_1}$ and $x_{i_n}$ are connected by an edge. The size of the cycle is the size of $C$.

A cycle $C$ is called induced if $C$ is not a cycle in the graph formed by the removal of any edge from the induced subgraph on $C$. If a cycle is not induced, it is called non-induced.
The theorem is:
\begin{Theorem}\label{d} (Theorem 1.1 \cite{FHT})
Let $G$ be a simple graph (a 2-hypergraph). A prime ideal $P=(x_{i_1},...,x_{i_r})$ is an associated prime of $(I(G)^\vee)^2$ if and only if:
\begin{enumerate}
\item [1.]$P=(x_{i_1},x_{i_2})$ and $\{x_{i_1},x_{i_2}\}$ is an edge in $G$, or
\item [2.]$r$ is odd and the set $\{x_{i_1},...,x_{i_r}\}$ is an induced cycle in $G$ for some ordering.
\end{enumerate}
\end{Theorem}

We observe that the primes $P$ in 1 are the minimal primes of $(I^\vee)^2$ (and $(I^\vee)$) and that a graph $G$ is bipartite if and only if it has no induced cycles of odd size. We thus recover Theorem \ref{c} from Theorem \ref{d}.

We can easily obtain the following result from Theorem \ref{d}:
\begin{Theorem}\label{thm}
Let $P\subset K[x_1,...,x_n]$ be a monomial prime. Then $P$ is an associated prime of $(I(G)^\vee)^2$ for some 2-hypergraph $G$ of size $n$ if and only if $\mbox{height}(P)\geq 2$ or $\mbox{height}(P)$ is an odd integer greater than or equal to 3.
\end{Theorem}

In the case of an m-hypergraph $H$, the minimal primes of $(I^\vee)^2$ are the primes $(x_{i_1},...,x_{i_m})$ such that $\{x_{i_1},...,x_{i_m}\}$ is an edge of $H$, as follows from equation (\ref{*}). Thus statement 1 of Theorem \ref{d} generalizes immediately to m-hypergraphs.
However, Theorem \ref{thm} and statement 2 of Theorem \ref{d} are only applicable to 2-hypergraphs. In contrast with Theorem \ref{thm}, we have the following theorem for 3-hypergraphs.

\begin{Theorem}\label{thm2}
Let $P\subset K[x_1,...,x_n]$ be a monomial prime. Then $P$ is an associated prime of $(I(H)^\vee)^2$ for some 3-hypergraph $H$ of size $n$ if and only if $\mbox{height}(P)\geq3$.
\end{Theorem}

The proof of this theorem will be given after Theorem \ref{t2}, which is the essential case of the proof.
\begin{Theorem}\label{t2}
For all $n\geq 3$, there exists a 3-hypergraph of size $n$ with edge ideal $I$ such that there is an associated prime of $(I^\vee)^2$ of height $n$.
\end{Theorem}

%An algorithm for constructing such a graph is as follows:

%Suppose $n=3$. Then the edge ideal $I=(x_1x_2x_3)$ satisfies the conditions.

%Suppose $n$ is an odd number of the form $4t+1$ greater than 3. Then let 
%\begin{equation}\label{eq1}
%I=(x_1x_2x_3,x_3x_4x_5,x_5x_6x_7,...,x_{n-2}x_{n-1}x_n,x_nx_1x_2)
%\end{equation}
%be the edge ideal for a graph. Then $(x_1,x_2,...,x_n)$ is an associated prime of $(I^\vee)^2$.

%Suppose $n$ is an odd number of the form $4t+3$ greater than 3. Then let
%\begin{equation}\label{eq4}
%I=(x_1x_2x_3,x_3x_4x_5,x_5x_6x_7,...,x_{n-2}x_{n-1}x_n,x_nx_1x_2,x_4x_5x_6)
%\end{equation}
%be the edge ideal for a graph. Then $(x_1,x_2,...,x_n)$ is an associated prime of $(I^\vee)^2$.

%Suppose $n$ is even, and $n = 4k+2$ for some integer $k$. Then let
%\begin{equation}\label{eq2}
%I=(x_1x_2x_3,x_3x_4x_5,x_5x_6x_7,...,x_{n-3}x_{n-2}x_{n-1},x_{n-1}x_nx_1)
%\end{equation}
%be the edge ideal for a graph. Then $(x_1,x_2,...,x_n)$ is an associated prime of $(I^\vee)^2$.

%Suppose $n$ is even, and $n=4k$ for some integer $k$. Then let
%\begin{equation}\label{eq3}
%I=(x_1x_2x_3,x_3x_4x_5,x_5x_6x_7,...,x_{n-3}x_{n-2}x_{n-1},x_{n-1}x_nx_1,x_2,x_3,x_4)
%\end{equation}
%be the edge ideal for a graph. Then $(x_1,x_2,...,x_n)$ is an associated prime of $(I^\vee)^2$.

\begin{proof}
We will prove the theorem in 4 cases.

Suppose $n$ is an odd number of the form $4t+1$ greater than 3.
Let $H$ be the 3-hypergraph with vertex set $V$, edge set $E$ and edge ideal $I$ such that
\begin{equation}\label{eq1}
I=(x_1x_2x_3,x_3x_4x_5,x_5x_6x_7,...,x_{n-2}x_{n-1}x_n,x_nx_1x_2).
\end{equation}
Let $S=\{x_2,x_4,x_6,...,x_{n-1}\}$ where $x_i\in S$ if and only if $i\leq n$ is even. Clearly, $S$ is an independent set. Further, $N(S)=\emptyset$, so $A_s=(a_1,a_2,...,a_n)$ where $a_i=1$ if $i$ is odd, and $a_i=0$ otherwise (we use the notation preceding Proposition \ref{prop5}). We will now show that $A_s$ is not the sum of two 1-covers. Suppose $A_s=Q+P$ where $Q=(q_1,...,q_n)$ and $P=(p_1,...,p_n)$ are 1-covers. Without loss of generality, assume $q_1=1$, $p_1=0$. Since $a_i=0$ for $i$ even, $q_i=p_i=0$ for $i$ even. For all odd $k<n$, $\{x_k,x_{k+1},x_{k+2}\}$ is in $E$. Thus if $k<n$ is odd, and $q_k=1$, we have $p_k=0$, $q_{k+2}=0$, and $p_{k+2}=1$. Similarily, if $p_k=1$, we have $q_k=0$, $p_{k+2}=0$, and $q_{k+2}=1$. We will show by induction that $q_i=1$ if and only if $i$ is of the form $4t+1$. Suppose $q_i=1$ if and only if $i$ is of the form $4n+1$ for all $i\leq 4j+1$. Then we have $q_{4j+2}=q_{4j+4}=0$ since $4j+2$ and $4j+4$ are even, and $q_{4j+3}=0$ since $q_{4j+1}=1$ and $4j+1$ is odd. Further, $p_{4j+3}=1$. Thus $q_{4j+5}=1$. $4j+5=4(j+1)+1$. Thus by induction, $q_i=1$ if and only if $i=4t+1$ for some $t$. Further, since $A_s=Q+P$, $p_i=1$ if and only if $i=4t+3$ for some $t$. Since $n$ is of the form $4t+1$, $p_i=0$ for $i\in\{n,1,2\}$. But $\{x_n,x_1,x_2\}\in E$. Thus $P$ is not a 1-cover, which is contradiction. Thus we see that $A_s$ is not the sum of two 1-covers, and so $\overline{A_s}$ is not in $(I^\vee)^2$. Now we must show that $((I^\vee)^2:\overline{A_s})=(x_1,x_2,...,x_n)$. This is equivalent to proving that $x_i\overline{A_s}\in(I^\vee)^2$ for all $i\leq n$. Let $k\leq n$. Suppose $k$ is odd. Then let $Q=(q_1,...,q_n)$ where $q_1=1$ and $q_i=1$ for all odd $i$ of the form $4t+1$ less than or equal to $k$ and $q_i=1$ for all odd $i$ of the form $4t+3$ greater than or equal to $k$, and $q_i=0$ everywhere else. Let $P=(p_1,...,p_n)$ where $p_i=1$ for all odd $i$ of the form $4t+3$ less than or equal to $k$, $p_i=1$ for all odd $i$ of the form $4t+1$ greater than or equal to $k$ , and $p_i=0$ everywhere else. Then it is clear that $x_k\overline{A_s}=\overline{Q+P}$. For all odd $i$, either $q_i=1$ and $p_{i-2}=1$ or vice versa. Also, $p_n=1$, so $Q$ and $P$ are 1-covers. Suppose $k$ is even, and of the form $4t$. Let $Q=(q_1,...,q_n)$ where $q_k=1$, $q_1=1$ and $q_i=1$ for all odd $i$ of the form $4t+1$ less than or equal to $k$ and $q_i=1$ for all odd $i$ of the form $4t+3$ greater than or equal to $k$, and $q_i=0$ everywhere else. Let $P=(p_1,...,p_n)$ where $p_i=1$ for all odd $i$ of the form $4t+3$ less than or equal to $k$, $p_i=1$ for all odd $i$ of the form $4t+1$ greater than or equal to $k$ , and $p_i=0$ everywhere else. Then it is clear that $x_k\overline{A_s}=\overline{Q+P}$. Then for $i\neq k+1$, either $q_i=1$ and $p_{i-2}=1$ or vice versa. Further, for $i=k+1$, $p_i=p_{i-2}=1$, but $q_k=1$. Thus $Q$ and $P$ are 1-covers. Finally, suppose $k$ is of the form $4t+2$. Let $Q$ and $P$ be as in the previous cases, with the difference that $q_k=0$ and $p_k=1$. Then once again, $Q$ and $P$ are 1-covers and $x_k\overline{A_s}=\overline{Q+P}$.
Thus $x_k\overline{A_s}\in (I^\vee)^2$. Therefore $(x_1,x_2,...,x_n)=((I^\vee)^2:\overline{A_s})$ is an associated prime of $(I^\vee)^2$.

Now suppose $n$ is an odd number of the form $4t+3$ greater than 3.
Let
\begin{equation}\label{eq4}
I=(x_1x_2x_3,x_3x_4x_5,x_5x_6x_7,...,x_{n-2}x_{n-1}x_n,x_nx_1x_2,x_4x_5x_6).
\end{equation}
Let $S=\{x_2,x_5,x_8,x_{10},x_{12},...,x_{n-1}\}$, so $x_i\in S$ if and only if $i=5$, $i=2$, or $i$ is an even number greater than or equal to 8 and less than or equal to $n-1$. $S$ is clearly an independent set, and $N(S)=\emptyset$. We will show that $A_s=(a_1,...,a_n)$ where $a_i=0$ if $x_i\in S$ and $a_i=1$ if $x_i\notin S$ is not the sum of two 1-covers. Suppose $A_s=Q+P$ where $Q=(q_1,...,q_n)$ and $P=(p_1,...,p_n)$ are 1-covers. Without loss of generailty, suppose $q_1=1$. Then we must have $p_3=1$, $q_4=1$, $p_6=1$, and $q_7=1$ since $Q$ and $P$ are 1-covers. Now we may use the argument in the proof of part 1 to say that if $i$ is greater than 7, $q_i=1$ if and only if $i$ is of the form $4t+3$, and $p_i=1$ if and only if $i$ is of the form $4t+1$. However, we now have $p_i=0$ for $i\in \{n,1,2\}$ since $n$ is of the form $4t+3$. But $\{x_n,x_1,x_2\}\in E$, so $P$ is not a 1-cover. Thus $A_s$ cannot be written as the sum of two 1-covers. Thus $\overline{A_s}\notin (I^\vee)^2$. Now we must show that $x_i\overline{A_s}\in (I^\vee)^2$ for $1\leq i \leq n$, which once again means we must show that $x_i\overline{A_s}$ corresponds to a sum of two 1-covers. Let $Q$ and $P$ be as just defined, so $Q+P=A_s$, but $P$ is not a 1-cover. Let $k\leq n$. Let $F=(f_1,...,f_n)$, $G=(g_1,...,g_n)$, where $f_i=q_i$ for $i<k$, $f_i=p_i$ for $i>k$, and $g_i=p_i$ for $i<k$ and $g_i=q_i$ for $i>k$. and $f_k=g_k=1$ if $a_k=1$, $f_k=1$, $g_k=0$ if $p_{k-1}=1$, and vice versa if $q_{k-1}=1$. Then we will have $\overline{F+G}=x_k\overline{A_s}$ and $F$ and $G$ are 1-covers. Thus $x_k\overline{A_s}\in(I^\vee)^2$, and so $(x_1,...,x_n)$ is an associated prime of $(I^\vee)^2$.

Now suppose $n$ is an even number, and $n=4z+2$ for some $z\geq 1$.
Let
\begin{equation}\label{eq2}
I=(x_1x_2x_3,x_3x_4x_5,x_5x_6x_7,...,x_{n-3}x_{n-2}x_{n-1},x_{n-1}x_nx_1).
\end{equation}
Let $S=\{x_2,x_4,...,x_{n}\}$, where $x_i\in S$ if and only if $i\leq n$ is even. $S$ is clearly independent, and $N(S)=\emptyset$. We will show that $A_s=(a_1,...,a_n)$ is not the sum of two 1-covers. Note that all the components of $A_s$ are either 1 or 0. Further, note that any vertex in the hypergraph is included in at most 2 edges. Thus if $W=(w_1,...,w_n)$ is some element of $\textbf{N}^n$ whose components are either 0 or 1 and has $q$ ones, then at most $2q$ edges $\{x_i,x_j,x_k\}$ have the property $w_i+w_j+w_k\geq 1$. Thus in order for an element of $\textbf{N}^n$ whose components are all 0 or 1 to be a 1-cover, at least $\frac{n}{4}$ components must be 1 since our hypergraph has $\frac{n}{2}$ edges. Since $n=4z+2$, this means that a 1-cover must have at least $z+1$ non-zero components. $A_s$ has $\frac{n}{2}=2z+1$ components whose value is 1. Thus if $A_s=Q+P$ where $Q$ is a 1-cover, $P$ can have at most $z$ non-zero components, and so $P$ is not a 1-cover. Therefore $A_s$ is not a sum of 1-covers, and $\overline{A_S}\notin (I^\vee)^2$.
Now we must show that $x_k\overline{A_s}\in(I^\vee)^2$ for all $k\leq n$. Suppose $k$ odd. Let $Q=(q_1,...,q_n)$, $P=(p_1,...,p_n)$ where $q_i=1$ for odd $i<k$ of the form $4t+1$, $q_i=1$ for odd $i>k$ of the form $4t+3$, $q_k=1$ and all other $q_i=0$. Similarily, let $p_i=1$ for odd $i<k$ of the form $4t+3$, $p_i=1$ for odd $i>k$ of the form $4t+1$, $p_k=1$ and all other $p_i=0$. Then it is clear that $\overline{Q+P}=x_k\overline{A_s}$. Further, $Q$ and $P$ are 1-covers. Suppose $k$ is even. Let $Q$ and $P$ be as before, with the alteration that $q_k=1$ if only if $k$ is of the form $4t$, and $p_k=1$ if only if $k$ is of the form $4t+2$. Then once again, $\overline{Q+P}=x_k\overline{A_s}$ and $Q$ and $P$ are 1-covers. Thus $(x_1,...x_n)$ is an associated prime of $(I^\vee)^2$.

Now suppose $n$ is an even number, and $n=4z$ for some $z\geq 1$.
Let
\begin{equation}\label{eq3}
I=(x_1x_2x_3,x_3x_4x_5,x_5x_6x_7,...,x_{n-3}x_{n-2}x_{n-1},x_{n-1}x_nx_1,x_2x_3x_4).
\end{equation}
Let $S=\{x_3,x_6,x_8,...,x_n\}$, so $x_i\in S$ if and only if $i=3$ or $i$ is an even number greater than or equal to 6 and less than or equal to $n$. $S$ is clearly independent, and $N(S)=\emptyset$. Again, we will show $A_s=(a_1,...,a_n)$ is not a sum of 1-covers. Suppose $A_s=Q+P$ where $Q=(q_1,...,q_n)$ and $P=(p_1,...,p_n)$ are 1-covers. Without loss of generality, suppose $q_1=1$. Then we must have $p_2=1$, $q_4=1$ and $p_5=1$ in order for $Q$ and $P$ to be 1-covers. Now we may use an argument as in the previous cases to show that for $i$ larger than 5, $q_i=1$ if and only if $i$ is of the form $4n+3$ and $p_i=1$ if and only if $i$ is of the form $4n+1$. However, $p_i=0$ for $i\in\{n-1,n,1\}$ since $n=4z$. Since $\{x_{n-1},x_n,x_1\}\in E$, $P$ is not a 1-cover. Thus $A_s$ is not a sum of two 1-covers. Now we show that $x_k\overline{A_s}\in(I^\vee)^2$ for all $k\leq n$. Let $F=(f_1,...,f_n)$, $G=(g_1,...,g_n)$ be defined as follows. Let $F$ be such that $f_i=q_i$ for $i<k$, $f_i=p_i$ for $i>k$, $f_k=1$ if $a_k=1$ or $p_{k-1}=1$ and $f_k=0$ otherwise. Similarily, let $G$ be such that $g_i=p_i$ for $i<k$, $g_i=q_i$ for $i>k$, $g_k=1$ if $a_k=1$ or $q_{k-1}=1$ and $g_k=0$ otherwise. Then $F$ and $G$ are 1-covers and $\overline{F+G}=x_k\overline{A_s}$. Thus $(x_1,...,x_n)$ is an associated prime of $(I^\vee)^2$.
\end{proof}

Now we will prove Theorem \ref{thm2}.
Suppose a monomial prime $P\subset K[x_1,...,x_n]$ is an associated prime of $(I(H)^\vee)^2$ for some 3-hypergraph $H$. Since the minimal associated primes of $(I(H)^\vee)^2$ are of height 3, $\mbox{height}(P)\geq3$.
Suppose $P\subset K[x_1,...,x_n]$ is a monomial prime with $\mbox{height}(P)\geq 3$. Then $P=(x_{i_1},...,x_{i_r})$ for some $1\leq i_1<\cdots<i_r\leq n$. By Theorem \ref{t2}, there exists a 3-hypergraph $H$ of size $r$ such that the ideal $Q=(x_{i_1},...,x_{i_r})\subset K[x_{i_1},...,x_{i_r}]$ is an associated prime of $(I(H)^\vee)^2$. Let $H'$ be the 3-hypergraph with vertex set $\{x_1,...,x_n\}$ and the same edge set as $H$. Then, because they have the same generators, $(I(H')^\vee)^2=(I(H)^\vee)^2 K[x_1,...,x_n]$. Further, $P=K[x_1,...,x_n]Q$, so $P$ is an associated prime of $(I(H')^\vee)^2$. Thus we have a 3-hypergraph of size $n$ such that $P$ is an associated prime of the square of the Alexander dual of the edge ideal.
\vskip.2truein

A hypergraph is connected if for any two vertices $x_i$ and $x_j$, there exists an ordered set of vertices $\{x_{p_1},...,x_{p_q}\}$ such that $x_{p_1}=x_i$, $x_{p_q}=x_j$ and $x_{p_k}$ is connected by an edge to $x_{p_{k+1}}$ for all $k<q$.

\begin{Corollary}\label{cor2}
For any integers $n$ and $q$, $n\geq q\geq3$, there exists a connected 3-hypergraph of size $n$ such that the square of the Alexander dual of the edge ideal has an associated prime of height $q$.
\end{Corollary}
\begin{proof}
In all cases of the proof of Theorem \ref{t2}, the hypergraphs used were connected. Thus for any $q\geq3$ we have a connected 3-hypergraph of size $q$ such that the square of the Alexander dual of the edge ideal has an associated prime of height $q$. Let $H'$ be such a hypergraph, with the vertex set of $H$ equal to $\{x_1,...,x_q\}$. Then add the vertices $\{x_{q+1},...,x_{n}\}$ and the edges $\{x_1,x_2,x_{q+1}\}$, $\{x_1,x_2,x_{q+2}\}$,...,$\{x_1,x_2,x_{n}\}$. Call this new graph $H$. Then $H'$ is an induced subhypergraph in $H$. Thus by Theorem \ref{indt}, $(I(H)^\vee)^2$ must have an associated prime of height $q$.
\end{proof}

Theorem \ref{thm2} also generalizes very nicely to $m$-hypergraphs where $m\geq 3$.

\begin{Theorem}\label{multd}
Let $P\subset K[x_1,...,x_n]$ be a monomial prime, $n\geq m \geq 3$ an integer. Then $P$ is an associated prime of $(I(H)^\vee)^2$ for some $m$-hypergraph $H$ of size $n$ if and only if $\mbox{height}(P)\geq m$.
\end{Theorem}

\begin{proof}
Let $P\subset K[x_1,...,x_n]$ be a monomial prime, $m \geq 3$, $P$ an associated prime of $(I(H)^\vee)^2$ for some $m$-hypergraph $H$. Then, since the minimal associated primes of $(I(H)^\vee)^2$ are of height $m$, height$(P)\geq m$.

To prove the rest of the theorem, we will first show that for any $q\geq m$, there exists an $m$-hypergraph $H$ of size $q$ such that $(I(H)^\vee)^2$ has an associated prime of height $q$.

Let $q\geq m$. Then $q-m+3\geq 3$. Thus, by Corollary \ref{cor2} there exists a 3-hypergraph $H$ of size $q-m+3$ such that the the maximal monomial prime is an associated prime of $(I(H)^\vee)^2$. Let $V=\{x_1,...,x_{q-m+3}\}$ be the vertex set of $H$ and $E$ be the edge set of $H$. By Theorem \ref{finder}, there exists a 2-cover of $H$, $C=(a_1,...,a_{q-m+3})$ such that $C$ is not the sum of two 1-covers, but for all $x_i$, $1\leq i \leq q-m+3$, the 2-cover corresponding to $x_i\overline{C}$ is a sum of two 1-covers. Define $H'$ to be the $m$-hypergraph with vertex set $V'=\{x_1,...,x_q\}$ and edge set $E'$ consisting of the sets $S'$ of vertices satisfying $S'\cap \{x_1,...,x_{q-m+3}\}\in E$ and $S'\cap \{x_{q-m+4},...,x_q\}=\{x_{q-m+4},...,x_q\}$. Suppose that $T=(t_1,..,t_{q-m+3})$ is a $k$-cover of $H$. Define a $q$-tuple $T'=(t_1,...,t_{q-m+3},0,...,0)$. $T'$ is a $k$-cover of $H'$. We shall say that $T'$ is the extension of $T$. Also, if $X$, $Y$, and $Z$ are $k$-covers of $H$ such that $X+Y=Z$, it is clear that $X'+Y'=Z'$. Thus $C'$ is a 2-cover of $H'$. Further, $C'$ is not the sum of two 1-covers. For $i\leq q-m+3$, the extension of the 2-cover corresponding to $x_i\overline{C}$ corresponds to $x_i\overline{C'}$. Thus, for $i\leq q-m+3$, $x_i\overline{C'}$ corresponds to a sum of two 1-covers. Further, for $i>q-m+3$, the $q$-tuple $(0,...,0,1,0,...,0)$ with a $1$ in the $i$th component is a 1-cover of $H'$ since $x_i\in S'$ for all $S'\in E'$. Thus $x_i\overline{C'}$ corresponds to a sum of two 1-covers for all $x_i$. Therefore $(I(H')^\vee)^2$ has an associated prime of height $q$. We note here that if $H$ is connected, $H'$ is connected.

Now suppose that $P=(x_{i_1},...,x_{i_q})\subset K[x_1,...,x_n]$ is a monomial prime, with height$(P)\geq m\geq 3$, $W$ an $m$-hypergraph of size $n$ such that the induced sub-hypergraph on $\{x_{i_1},...,x_{i_q}\}$ has an associated prime of size height$(P)$.

Such an $m$-hypergraph $W$ can be constructed as follows: Let $Z$ be an $m$-hypergraph of size $q$=height$(P)$ with vertex set $V=\{x_{i_1},...,x_{i_q}\}$ and edge set $E$ such that $P'=(x_{i_1},...,x_{i_q})\subset K[x_{i_1},...,x_{i_q}]$ is an associated prime of $(I(Z)^\vee)^2$. Let $W$ be the $m$-hypergraph with vertex set $V'=\{x_1,...,x_n\}$ and edge set $E'$ such that $$E=E'\cup\{\{x_{i_1},...,x_{i_{m-1}},x_j\}\mid x_j\in V'-V\}.$$ We note that if $Z$ is connected, $W$ is connected. Then $Z$ is the induced sub-hypergraph of $W$ on $\{x_{i_1},...,x_{i_q}\}$, so by Theorem \ref{indt} $P$ is an associated prime of $(I(W)^\vee)^2$.
\end{proof}

Note Corollary \ref{cor2} also extends to $m$-hypergraphs, that is
\begin{Corollary}\label{cor3}
For any integers $n$, $q$, and $m$, $n\geq q\geq m\geq 3$, there exists a connected $m$-hypergraph of size $n$ such that the square of the Alexander dual of the edge ideal has an associated prime of height $q$.
\end{Corollary}
\begin{proof}
By Corollary \ref{cor2}, we have a connected 3-hypergraph $H$ of size $q-m+3$ with an associated prime of height $q-m+3$. We can extend this to an $m$-hypergraph $H'$ of size $q$ with an associated prime of height $q$ using the method in the proof of Theorem \ref{multd}. Since $H$ is connected, $H'$ is connected. By the method in the last paragraph of the proof of Theorem \ref{multd}, we can extend $H'$ to an $m$-hypergraph $H''$ of size $n$ with an associated prime of height $q$. Further, since $H'$ is connected, $H''$ will be connected as well.
\end{proof}

\par\vskip.5truecm\noindent
Ashok Cutkosky  \par\noindent Hickman High School\par\noindent 1104 N Providence rd.\par\noindent Columbia, MO 65203
\par\noindent
imahtarmaca@gmail.com\par\noindent

\end{document}